\documentclass{siamltex}

\usepackage{amsmath}
\usepackage{amssymb}
\usepackage{tikz}

\newtheorem{thm}{Theorem}
\newtheorem{cor}{Corollary}
\newtheorem{lem}{Lemma}
\newtheorem{defn}{Definition}
\def\beq{\begin{equation}}
\def\eeq{\end{equation}}
\def\com#1{\quad{\textrm{#1}}\quad}
\def\eq#1{(\ref{#1})}
\def\nn{\nonumber}
\def\x#1{$#1\times#1$}
\def\barx{\overline{x}}
\def\tilx{\tilde{x}}
\def\tilt{\tilde{t}}
\def\bp{\backprime}
\def\RC{$R/C$ }
\def\dbyd#1{\textstyle{\frac{\partial}{\partial#1}}}

\begin{document}

\title{Smooth solutions and singularity formation for the
  inhomogeneous nonlinear wave equation\thanks{Parts of this paper are
    included in this author's Ph.D dissertation at the University of
    Massachusetts, Amherst.}}

\author{Geng Chen\thanks{Department of Mathematics, Pennsylvania State
    University, University Park, PA, 16802 ({\tt chen@math.psu.edu}).}
  \and Robin Young\thanks{Department of Mathematics and Statistics,
    University of Massachusetts, Amherst, MA 01003 ({\tt
      young@math.umass.edu}).  Supported in part by NSF Applied
    Mathematics Grant Number DMS-0908190.}}

\pagestyle{myheadings}
\thispagestyle{plain}
\markboth{GENG CHEN and ROBIN YOUNG}{SINGULARITY FORMATION
FOR THE WAVE EQUATION}

\date{}

\maketitle

\begin{abstract}
  We study the nonlinear inhomogeneous wave equation in one space
  dimension: $v_{tt} - T(v,x)_{xx} = 0$. By constructing some
  ``decoupled" Riccati type equations for smooth solutions, we provide
  a singularity formation result without restrictions on the total
  variation of unknown, which generalize earlier singularity results
  of Lax and the first author. These results are applied to several
  one-dimensional hyperbolic models, such as compressible Euler flows
  with a general pressure law, elasticity in an inhomogeneous medium,
  transverse MHD flow, and compressible flow in a variable area duct.
\end{abstract}



\section{Introduction}

In this paper, we consider the initial value problems for the
second-order quasilinear inhomogeneous wave equation in one space
dimension,
\begin{equation}
  v_{tt} - T(v,x)_{xx} = 0,
\label{inhom nl wave}
\end{equation}
where $(x,t)\in\textbf{R}\times\textbf{R}^+$, $v(x,t)\in\textbf{R}$,
and $T(v,x)$ is a smooth function satisfying
\[
   T_v >0, \quad T_{v v }<0.
\]
These assumptions imply that the equation is hyperbolic and genuinely
nonlinear, so that solutions exhibit wave-like behavior.  Signals
propagate in a forward and backward direction with local nonlinear
wavespeeds
\[
   c = \sqrt{T_v} \com{and} -c = - \sqrt{T_v},
\]
respectively.  Equation (\ref{inhom nl wave}) includes a wide variety
of interesting physical systems, such as one-dimensional nonlinear
elasticity in an inhomogeneous medium \cite{Fritz John,Fritz
  John0}. In the context of elasticity, $x$ is a material coordinate,
$v$ is the strain and $T$ is the elastic stress.

It is well known that solutions of nonlinear hyperbolic equations
generally form shock waves in finite time.  Shocks form as a result of
gradient blowup, which is a consequence of genuine nonlinearity.  In
this paper, we study smooth solutions, and we are particularly
interested in the lifetime of such smooth solutions.  We study the
dynamical system which governs the growth of gradients, and use this
to give estimates on the lifetimes of smooth solutions.  This was
carried out for the homogeneous nonlinear wave equation, obtained by
taking $T=T(v)$, by Lax in~\cite{lax2}, and extensions to larger
systems (with restrictions on the data) were obtained by
John~\cite{Fritz John} and Liu~\cite{Liu1}.  

It is convenient to write \eq{inhom nl wave} as a system, by setting
\[
  u(x,t)=\int{v _t}\;dx,\quad
  p(v ,x)= -T,
\]
which yields the first-order system
\begin{align}
v _t-u_x&=0,\label{GE lagrangian1}\\
u_t+p(v ,x)_x&=0,\label{GE lagrangian2}
\end{align}
with 
\beq
   p_{v }<0, \quad p_{v v }>0.
\label{GE p condition0}
\eeq
We make the further assumption that $p=p(v,x)$ is a $C^3$ function,
and that boundedness of derivatives of $p$ and the sign conditions
\eq{GE p condition0} are uniform in $x$.  In particular, $p_x$ and
$p_{xx}$ are uniformly bounded in $x$.

For smooth solutions, equations \eq{GE lagrangian1}, \eq{GE
lagrangian2} model compressible inviscid flow for general pressure
laws in material coordinates (the Euler equations in a Lagrangian
frame).  Here $p$ is the pressure, $v$ is the specific volume, and $u$
is the fluid velocity.  A general pressure is described by $p=p(v,S)$,
but the entropy $S=S(x)$ is stationary as long as the flow is smooth,
by the Second Law of Thermodynamics.  Moreover, smooth solutions of
the $5\times5$ system of one-dimensional transverse flow in
Magnetohydrodynamics (MHD), which models a fluid coupled to a
transverse (magnetic) vector field, and Eulerian flow in a variable
area duct can also be modeled by (\ref{GE lagrangian1}), (\ref{GE
lagrangian2}).

For homogeneous isentropic flow, we have $p=p(v)$, and (\ref{GE
lagrangian1}), (\ref{GE lagrangian2}) is known as the the
$p$-system.  This is a simplified \x2 system which admits a coordinate
system of Riemann invariants, and for which Lax proved that any
nontrivial data will form a shock wave in finite time~\cite{lax2}.
F. John and T.-P. Liu extended Lax's results to general systems of
conservation laws of the form
\beq
  u_t+f(u)_x=0,
\label{conservation laws}
\eeq
with $u=(u_1,\cdots,u_n)\in\textbf{R}^n$, for $n\geqslant1$ and
$(x,t)\in\textbf{R}\times\textbf{R}^+$.  John proved that gradients
will blow up for small compactly supported data, and this was later
generalized by Liu, provided the total variation of the data is small
enough~\cite{Fritz John,Liu1}.   There are also some singularity
results for multi-dimensional conservation laws, subject to a
restrictive ``null condition''~\cite{Rammaha, sideris}.  In a
recent paper \cite{G3}, the first author generalizes the singularity
formation results in \cite{lax2} to the $3\times3$ compressible
Euler equations with polytropic ideal gas.

In this paper, we give singularity formation results for (\ref{GE
  lagrangian1}), (\ref{GE lagrangian2}), or equivalently
(\ref{inhom nl wave}).  In~\cite{Fritz John,Liu1}, small solutions are
expanded asymptotically along integral curves, and wave interactions
are treated quadratically.  Here we take a different point of view,
considering three wave families, namely forward and backward waves, as
in the $p$-system, coupled with stationary waves, which carry entropy
variations and inhomogeneity.  We are then able to treat all waves
without regard to their wave strength.

Our first task is to define the rarefactive and compressive character
(\RC character) of the nonlinear (non-stationary) waves in smooth
solutions.  In a \x2 system, which is diagonal when expressed in
Riemann invariants, it is clear when a wave is rarefying or
compressing.  However, in larger systems, this distinction is not
clear as waves of different families generally interact continuously
and cannot be decoupled .  The \RC character is a quantitative measure
of how much rarefaction or compression is in the solution at any
point.

In stationary solutions (including those with contact
discontinuities), which include no compression or rarefaction, the
pressure $p$ and velocity $u$ are constant, by the Rankine-Hugoniot
conditions.  We therefore use changes in $p$ (or $u$) to define the
\RC character of the solution.  In doing so, we take the directional
derivative \emph{along the opposite characteristic}, to minimize the
effect of waves of the opposite family, see \cite{G3}.  We use the
superscripts $``\backprime"$ and $``\prime"$ to denote the directional
derivatives along backward and forward characteristics, respectively,
so that
\beq
  {}^{\backprime}=\partial_t-c\, \partial_x \com{and}
  {}^{\prime}=\partial_t+c\, \partial_x,
\label{GE prime backprime}
\eeq
where $c = \sqrt{-p_v}$ is the (local) wavespeed in Lagrangian
coordinates, c.f.~\cite{lax2}.

\begin{defn}
\label{GE def1}
If the solutions of (\ref{GE lagrangian1}), (\ref{GE lagrangian2})
are smooth in an open set $U$ of the (x,t)-plane and $A$ is a point in
$U$, then we say the solution is forward (backward) rarefactive at
$A$, if and only if $p^{\backprime}<0$ ($p^{\prime}<0$); it is forward
(backward) compressive at $A$, if and only if $p^{\backprime}>0$
($p^{\prime}>0$).
\end{defn}

It is convenient to introduce the change of variables
\beq
   h(v,x) \equiv\int_v^{v ^{*}} \sqrt{-p_v}\;dv,\com{and}
   \mu \equiv x,
\label{hmudef}
\eeq
where $v^*$ is a constant or infinity.  Then by
calculating $p^{\backprime}$ and $p^{\prime}$, we introduce
equivalent variables $\alpha$ and $\beta$, defined by
\begin{align*}
  \alpha &\equiv -\frac{p^{\backprime}}{c^2}
        =u_x+ h_x+\frac{p_\mu}{c},\\
  \beta &\equiv -\frac{p^{\prime}}{c^2}
        =u_x-h_x-\frac{p_\mu}{c}.
\end{align*}
Here $u\pm h$ are the Riemann invariants for the corresponding
isentropic $p$-system.  Thus $\alpha$ and $\beta$ are direct
generalizations of the derivatives of the Riemann
invariants.  For smooth solutions, we derive Riccati type
ODEs for $\alpha$ and $\beta$, which provide a framework for studying
smooth solutions and gradient blowup.

\begin{thm}
\label{GE remark}
In (\ref{GE lagrangian1}), (\ref{GE lagrangian2}), smooth
solutions satisfy 
\begin{align}
  \alpha^\prime&=-\frac{c}{2}(\frac{p_\mu}{c})_h
(3\alpha+\beta)+\frac{c_h}{2}(\alpha\beta-\alpha^2), \label{GE rem1}
\\
  \beta^\backprime&=\frac{c}{2}(\frac{p_\mu}{c})_h
(\alpha+3\beta)+\frac{c_h}{2}(\alpha\beta-\beta^2),  \label{GE rem2}
\end{align}
with $c_h>0$.
\end{thm}

Equations (\ref{GE rem1}) and (\ref{GE rem2}) can be
decoupled by use of an integrating factor.  Define
\begin{align}
  y &\equiv \sqrt{c}\alpha-I
    =\sqrt{c}(u+ h)_x+\frac{p_\mu}{\sqrt{c}}-I,
\com{and}\label{GE y def}\\
  q &\equiv \sqrt{c}\beta+I
    =\sqrt{c}(u-h)_x-\frac{p_\mu }{\sqrt{c}}+I,
\label{GE q def}
\end{align}
where
\beq
   I=I(h,\mu) \equiv 
   \int_{h_0}^{h}{\frac{1}{2}\sqrt{c}(\frac{p_\mu}{c})_h}\;dh,
\label{GE Idef}
\eeq
and $h_0$ is a constant.

\begin{thm}
\label{GE them new ODEs}
For smooth solutions of (\ref{GE lagrangian1}), (\ref{GE
lagrangian2}), we have
\begin{align}
  y^\prime&=a_0+a_1\,y - a_2\,y^2,
\label{GE new ode1}\\
  q^\backprime&=a_0-a_1\,q - a_2\,q^2,
\label{GE new ode2}
\end{align}
where
\begin{align}
   a_0 &\equiv -c I_\mu +
  \frac{1}{2}\sqrt{c}(\frac{p_\mu }{c})_h p_\mu -
  c(\frac{p_\mu}{c})_h I-\frac{c_h}{2\sqrt{c}}I^2,
\label{GE a0}\\
   a_1 &\equiv -(2\sqrt{c}I)_h,\label{GE a1}\\
   a_2 &\equiv \frac{c_h}{2\sqrt{c}} > 0.
\label{GE a2}
\end{align}
\end{thm}

We note that these are not closed ODEs, since both the directional
derivatives and the coefficients are dependent on the underlying
solution of (\ref{GE lagrangian1}), (\ref{GE lagrangian2}).
Nevertheless, as in \cite{lax2}, we are able to compare them to closed
ODEs and derive bounds on the lifespan of smooth solutions.  For
convenience, we only consider smooth, i.e. $C^\infty$, initial data,
although our results also apply to $C^2$ initial data.

\begin{thm}
\label{GE Thm singularity2}
Assume the smooth solution takes values in some compact set $\mathcal
K$, uniformly in $x$.  Then there is some constant $N>0$ depending
only on $\mathcal K$, such that, if $y$ or $q$ is less than $-N$
somewhere in the initial data, then $|u_x|$ and/or $|v _x|$ blow up in
finite time: that is, there is some critical time $T_*<\infty$, so
that
\[
  \max\left\{|u_x|,|v_x|\right\} \to \infty \com{as} t\to T_*.
\]
\end{thm}

Blowup of the gradient does not mean that the solution fails to exist:
rather, this usually heralds the formation of a shock, and the
associated decay of solutions.  Generally, solutions are continued as
weak solutions which contain shocks.  However, in order to study shock
propagation our system needs to be in conservation form, and we do not
consider those issues here.

Our assumption that the solution stay in a compact set $\mathcal K$
can be stated technically as 
\begin{gather}
  |h|<B_1, \quad A_2<c<B_2, \ A_3<c_h<B_3,
\label{GE assumption}\\\nonumber
  |c_\mu|<B_4, \quad |c_{\mu\mu}|<B_5, \quad 
  |c_{h\mu}|<B_6, \quad |p_\mu|<B_7, \quad |p_{\mu\mu}|<B_8,
\end{gather}
where $A_i$ and $B_i$ are all positive constants.  These conditions
simply mean that we do not leave the domain $\mathcal K$ in the
solutions we consider, and allow us to stay away from vacuum and other
points where the equation may become singular.  From a practical point
of view, the only restriction on the data (apart from smoothness
conditions) is that the solution stays away from vacuum: in general we
do not expect the pressure $p$ to grow unboundedly.  For the
$p$-system, an invariant region provides an upper bound for the
pressure, and although there is not an invariant region for the full
system, physically we expect the pressure to remain finite.  We note
that the assumption that the solution does not form a vacuum is
implicit in~\cite{lax2}, but this is reasonable as a vacuum cannot
form in finite time, see~\cite{young Global interaction,young com}.

Theorem \ref{GE Thm singularity2} implies that gradients of
solutions blow up if the initial compressions are strong enough.
When the variation of entropy is mild, $N$ is close to zero, so the
shock free solutions are ``almost rarefactive", which is consistent
with Lax's singularity formation results in \cite{lax2}. In
\cite{young blake 1} and a forthcoming paper \cite{G4}, examples of
solutions containing compressive waves are constructed, but the gradients
of those solutions remain finite.

In a series of recent papers \cite{young blake 1,young blake 2,young
  blake 3}, the possibility of time-periodic solutions in the
compressible Euler equations has been demonstrated.  A critical
feature of this study is how the \RC structure of waves can
change across a contact discontinuity.  In this paper we check the
consistency of our results, which presuppose a smooth entropy field,
with those results.  By studying the \RC structure and the way
in which it can change further, we expect eventually to see
time-periodic solutions as in \cite{young blake 1,young blake 2} with
piecewise smooth entropy, consisting of both contact discontinuities
and smooth entropy variations.

Our results apply directly to a number of systems having structure
similar to the inhomogeneous nonlinear wave equation.  In addition to
equations (\ref{GE lagrangian1}), (\ref{GE lagrangian2}), we apply
these ideas to transverse flow in one-dimensional magnetohydrodynamics
(MHD), a $5\times5$ system modeling fluid motion coupled to a
transverse magnetic field.  By restricting to a polytropic ideal gas,
we are able to give a stronger singularity formation result than that
of Theorem~\ref{GE Thm singularity2}, similar to that of \cite{G3}.
Finally, we consider inviscid compressible flow in a varying area
duct, to which similar ideas apply.

The paper is arranged as follows.  In Section~2, we give the
background of the equations and establish some useful identities.  In
Section~3, we define rarefactive and compressive waves.  In
Sections~4 and 5, we prove Theorems \ref{GE remark}--\ref{GE Thm
  singularity2}, respectively. In Section~6, we consider the
consistency of \RC structures, and in Sections~7 and~8, we apply our
ideas to one-dimensional transverse flow of MHD and compressible flow
in a varying duct.

\section{Coordinates and background}

We focus on equations (\ref{GE lagrangian1}), (\ref{GE
lagrangian2}) from now on.  We assume that $p(v ,x)$ is a smooth
function of $v $ and $x$, satisfying
\[
   p_{v }<0, \com{and} p_{v v }>0,
\]
for all $x$ and $v=v(x,t)\in(0,\infty)$.  These conditions imply
hyperbolicity and genuine nonlinearity, respectively.  The vacuum
state corresponds to $v=\infty$; however, since we assume the data
remains in a compact set, we will not address questions at vacuum.
The local absolute wavespeed is
\beq
   c(v ,x) \equiv \sqrt{-p_v}>0.
\label{GE c ptau}
\eeq
We make the change of variables
\[
  h(v ,x) \equiv \int_v ^{v^{*}} c\; dv 
     =\int_v ^{v ^{*}} \sqrt{-p_v}\;dv,
\]
see \eq{hmudef}, where $v^*>0$ is a convenient constant (or
$\infty$ if the integral converges uniformly).  Since $p(v,x)$ is
smooth, the function $h(v,x)$ is also smooth with respect to $v$ and
$x$.  Moreover, since
\beq 
  h_v (v ,x)=-c<0,
\label{GE h tau}
\eeq
the inverse function $v =v (h,\mu)$ is smooth with respect to $h$ and
$\mu$, where we have set
\beq
  \mu \equiv x.
\label{GE mu}
\eeq

For any function $f(v,x)$, we will write
\[
   f(h,\mu)=f(h(v ,x),\mu)=f(v ,x),
\]
without ambiguity, and we use the subscript notation
\begin{gather*}
  f_x = \dbyd{x}{f(v(x,t),x)},\quad
  f_{\barx} = \dbyd{x}{f(v,x)},\quad
  f_v = \dbyd{v}{f(v,x)},\\
  f_h = \dbyd{h}{f(h,\mu)},\quad
  f_\mu = \dbyd{\mu}{f(h,\mu)},
\end{gather*}
for the various partial derivatives of $f$.

We can relate the different partial derivatives as follows:
by (\ref{GE h tau}),
\beq 
   v _h=-\frac{1}{c},
\label{GE tau h}
\com{so that}
  v _t=v _h h_t=-\frac{1}{c} h_t.
\eeq
Furthermore, since
\beq
  -c^2=p_v =p_h h_v =-c p_h,
\com{we have} 
   p_h=c,
\label{GE p h}
\eeq
and so
\[
  p_x=p_h h_x+p_\mu=c h_x+ p_\mu.
\]
Thus, for smooth solutions, (\ref{GE lagrangian1}), (\ref{GE
  lagrangian2})  can be written as
\begin{align}
  h_t+c u_x&=0,
\label{GE lagrangian1 hm}\\\label{GE lagrangian2 hm}
  u_t+c h_x+p_\mu&=0.
\end{align}

Next, differentiating $v =v(h(v , x),\mu)$ with respect to $x$, and
recalling \eq{GE mu} and \eq{GE tau h}, we get
\[
  0=v _h h_{\barx}+v _\mu,
\com{so that}
  v _\mu=-v _h h_{\barx}=\frac{h_{\barx}}{c}.
\]
Similarly, for any function $f(v,x)$, we get
\begin{align}
  f_\mu &= f(v (h,\mu),x)_\mu\nn\\
  &= f_v v _\mu+f_{\barx}\nn\\
  &= \frac{f_v }{c}h_{\barx}+f_{\barx},\label{GE xiMtaux}
\end{align}
and
\beq 
  f_h=f_v v _h =-\frac{f_v }{c}.
\label{GE xiZtaux}
\eeq
It follows that, for smooth solutions, $p$ and $c$ are smooth with
respect to $h$ and $\mu$.  Moreover, using (\ref{GE xiMtaux}) and
(\ref{GE xiZtaux}), any quantities such as the \RC character,
ODEs, and singularity formation results in this paper can all be
expressed in the variables $(v ,x)$ instead of $(h,\mu)$.

\section{Compressive and Rarefactive waves}

In this section, we define the rarefactive ($R$) and compressive
($C$) characters of (\ref{GE lagrangian1}), (\ref{GE
lagrangian2}), which quantitively indicate the amount of rarefaction or
compression in the solutions at any point.

We first consider isentropic flow, for which $p=p(v)$ and our system
reduces to the $p$-system,
\begin{align}
  h_t+c u_x&=0,\label{GE p1 hm}\\
  u_t+c h_x&=0.\label{GE p2 hm}
\end{align} 
The Riemann invariants $s=u+h$ and $r=u-h$ satisfy the diagonal system
\[
  s_t+cs_x=0,\quad
  r_t-cr_x=0,
\]
and so are constant along forward and backward characteristics,
\[
  \frac{dx}{dt}=+ c,\quad
  \frac{dx}{dt}=- c,
\]
respectively.  In an isentropic domain, because the system is
diagonal, it is clear when waves are compressive or rarefactive:
indeed, the amount of compression or rarefaction can be measured by
derivatives of the appropriate Riemann invariant.  There are several
equivalent conditions, and for us it is convenient to consider the
change in pressure: if $p$ decreases as we traverse the wave from
front to back, the wave is rarefactive (\emph{R}), while if $p$
increases, it is compressive (\emph{C})~\cite{G3,lax2,young Global
  interaction}.  This is consistent with the entropy condition for
shocks, which states that the pressure is always larger behind a
shock. 

When $p=p(v,x)$ explicitly depends on $x$, we first consider
stationary solutions, in which there are no compressive or rarefactive
waves, so the \RC characters should vanish.  In stationary solutions,
the pressure $p$ is constant,
\[
  p_t = p_v\;v_t = 0 \com{and} p_x = -u_t = 0,
\]
so its directional directives are zero.  Physically, in gas dynamics,
this means that pressure is not impacted by the variation of entropy,
c.f.~\cite{G3}.  Thus, by considering the directional derivatives of
$p$ along the \emph{opposite} characteristics, we obtain
Definition~\ref{GE def1}: the solution is forward (backward)
rarefactive at $A$, if and only if $p^{\backprime}<0$
($p^{\prime}<0$); it is forward (backward) compressive at $A$, if and
only if $p^{\backprime}>0$ ($p^{\prime}>0$).

We could also use $u$ to define the \RC character, as this is
also constant in stationary solutions.  We define
\[
  \alpha \equiv -\frac{p^\backprime}{c^2} \com{and}
  \beta \equiv -\frac{p^\prime}{c^2}.
\]

\begin{lem}
\label{alphabeta ux hx}
For smooth solutions of (\ref{GE lagrangian1}), (\ref{GE
  lagrangian2}), we have
\[
  p^\prime = -c u^\prime \com{and}
  p^\backprime = c u^\backprime,
\]
while also
\beq
  \alpha=u_x+h_x+\frac{p_\mu}{c}  \com{and}
  \beta=u_x-h_x-\frac{p_\mu}{c}.
\label{GE alpha beta}
\eeq
\end{lem}

\begin{proof}
By (\ref{GE lagrangian1}), (\ref{GE lagrangian2}) and (\ref{GE c ptau}),
\[
  c u^\prime=c u_t+c^2 u_x=-cp_x-p_v v_t=-p^{\prime},
\]
and similarly $p^\backprime=cu^\backprime$.  By (\ref{GE p h}) and
(\ref{GE lagrangian1 hm}),
\begin{align*}
  -c^2 \beta &= p^\prime = p_t+c p_x\\ 
             &= p_h h_t +c(p_h h_x+p_\mu)\\
             &=-c^2(u_x-h_x-\frac{p_\mu}{c}),
\end{align*}
and similarly for $\alpha$, so \eq{GE alpha beta} follows.
\end{proof}

\begin{cor}
\label{GE def2}
The \RC character of a smooth solution is given by:
\[
\begin{array}{llll}
 Forward& R & \it{iff} & \alpha>0,\\
 Forward& C & \it{iff} & \alpha<0,\\
 Backward& R & \it{iff} & \beta>0,\\
 Backward& C & \it{iff} & \beta<0.
\end{array}
\]
Moreover, provided the solution values remain in $\mathcal K$,
\beq
  |\alpha|\ \text{or}\ |\beta| \to \infty \com{iff}
  |u_x|\ \text{or}\ |v _x| \to \infty.
\label{GE alpha beta blowup}
\eeq
\end{cor}

\begin{proof}
Clearly, by (\ref{GE alpha beta}),
\[
  p^\backprime \gtrless 0 \Leftrightarrow \alpha \lessgtr 0,
\com{and}
  p^\prime \gtrless 0 \Leftrightarrow \beta \lessgtr 0.
\]
By Lemma \ref{alphabeta ux hx},
\[
  \alpha+\beta=2u_x,\quad
  \alpha-\beta=2(h_x+\frac{p_\mu}{c}),
\]
and (\ref{GE alpha beta blowup}) follows since $p_\mu$ and $c$ remain
finite.
\end{proof}

In an isentropic domain, i.e. $p=p(v)$, it is clear that
\[
   s_{x}=\alpha \com{and} r_{x}=\beta,
\]
so we can regard $\alpha$ and $\beta$ as direct generalizations of the
derivatives of the Riemann invariants.

\section{Differential equations for gradients}

In this section, we consider the characteristic decompositions of
smooth solutions.  By considering the directional derivatives of
$\alpha$ and $\beta$, we derive the ODEs for $\alpha$ and $\beta$
as stated in Theorem~\ref{GE remark}.

{\it Proof of Theorem \ref{GE remark}}.\quad
We show (\ref{GE rem1}), since (\ref{GE rem2}) follows in exactly the
same way.  We have
\begin{align}
\alpha^\prime&=(u_x+ h_x+\frac{p_\mu}{c})_t+c(u_x+
h_x+\frac{p_\mu}{c})_x\nn\\\label{GE
app0}&=(u_{xt}+c h_{xx})+(h_{xt}+c u_{xx})
+(\frac{p_\mu}{c})_t+c(\frac{p_\mu}{c})_x\\\nn&=(u_{t}+c
h_{x})_x+(h_{t}+c u_{x})_x-c_x h_x-c_x
u_x+(\frac{p_\mu}{c})_t+c(\frac{p_\mu}{c})_x.
\end{align}
By (\ref{GE lagrangian1 hm}), (\ref{GE lagrangian2 hm}),
\[
  (u_t+c h_x)_x=(-p_\mu)_x=-p_{\mu h} h_x-p_{\mu \mu },
\]
and
\[
  (h_t+c u_x)_x=0.
\]
Thus the right hand side of (\ref{GE app0}) is
\beq
  -p_{\mu h} h_x-p_{\mu \mu }-
  (c_h h_x+c_\mu)(h_x+u_x)+(\frac{p_\mu }{c})_h(h_t+ c h_x)
  +c(\frac{p_\mu}{c})_\mu,
\label{GE app1}
\eeq
since $f_x=f_h h_x+f_\mu$ for any function $f$.

By (\ref{GE alpha beta}) and (\ref{GE lagrangian1 hm}), we have
\[
   h_x=\alpha-u_x-\frac{p_\mu }{c},
\com{and}
   h_t=-c u_x,
\]
and by (\ref{GE p h}), 
\beq
   c(\frac{p_\mu }{c})_h=c_\mu -c_h \frac{p_\mu}{c}.
\label{GE c pm c h}
\eeq
Thus (\ref{GE app1}) can be simplified to
\begin{align*}
-p_{\mu h}\;&
(\alpha-u_x-\frac{p_\mu }{c})-p_{\mu \mu }-(c_h \alpha-c_h u_x-c_h
\frac{p_\mu }{c}+c_\mu )(\alpha-\frac{p_\mu }{c})\\&\qquad\qquad{}+(\frac{p_\mu
}{c})_h( c \alpha-2c u_x-p_\mu )+c(\frac{p_\mu }{c})_\mu\\
&=-c(\frac{p_\mu }{c})_h u_x+[ c_h u_x -c(\frac{p_\mu }{c})_h]
\alpha-c_h \alpha^2.
\end{align*}
Finally, making the substitution
\[
   u_x=\frac{\alpha+\beta}{2}
\]
yields (\ref{GE rem1}).

By (\ref{GE p condition0}), (\ref{GE c ptau}) and (\ref{GE h tau}),
\beq
   0>(\sqrt{-p_v })_v =c_v =c_h h_v =c_h (-c),
\label{GE ChPositive}
\eeq
so $c_h>0$, and the theorem is proved.
$\square$\bigskip

\begin{cor}
\label{GE cor change coordinate}
For smooth solutions of (\ref{GE lagrangian1}), (\ref{GE
  lagrangian2}), we have
\[
   \alpha^\prime=-\frac{c}{4}(\frac{p_{\barx}}{p_v})_v
  (3\alpha+\beta)-\frac{c_v }{2c}(\alpha\beta-\alpha^2),
\]
and
\[
  \beta^\backprime=\frac{c}{4} (\frac{p_{\barx}}{p_v})_v
  (\alpha+3\beta)-\frac{c_v }{2c}(\alpha\beta-\beta^2).
\]
\end{cor}

\begin{proof}
By (\ref{GE c ptau}), (\ref{GE h tau}), (\ref{GE xiMtaux}), (\ref{GE
xiZtaux}) and (\ref{GE c pm c h}), we have
\begin{align}
  c(\frac{p_\mu }{c})_h&=c_\mu -c_h\frac{p_\mu }{c}\nn
    \\&=c_{\barx}+\frac{c_v }{c^2}p_{\barx} \nn\\
    &=\frac{c}{2}(\frac{p_{\barx}}{p_v })_v,
\label{GE RC function}
\end{align}
and by (\ref{GE ChPositive}),
\[
  \frac{c_h}{2}=-\frac{c_v }{2c},
\]
and the corollary follows from Theorem \ref{GE remark}.
\end{proof}

We now make another change of variables and transform (\ref{GE rem1})
and (\ref{GE rem2}) into decoupled differential equations by use of an
integrating factor.

{\it Proof of Theorem \ref{GE them new ODEs}}.\quad
First, the condition $a_2>0$ follows immediately from (\ref{GE
  ChPositive}).

By (\ref{GE alpha beta}) and (\ref{GE lagrangian1 hm}),
\[
  h^{\prime}=h_t+c h_x=-c u_x+ c h_x=-c(\beta+\frac{p_\mu }{c}),
\]
so that
\[
  \beta=-\frac{h^\prime}{c}-\frac{p_\mu }{c}.
\]
Hence, (\ref{GE rem1}) can be written
\[
  \alpha^\prime=-\frac{c}{2}(\frac{p_\mu }{c})_h
  (3\alpha-\frac{h^\prime}{c}-\frac{p_\mu }{c})+\frac{c_h}{2}\alpha
  (-\frac{h^\prime}{c}-\frac{p_\mu }{c})-\frac{c_h}{2}\alpha^2.
\]
We move the terms including $h^\prime$ to the left hand side, so
\begin{align}
  \alpha^{\prime}-\frac{1}{2}(\frac{p_\mu }{c})_h
&h^\prime+\frac{c_h}{2c}\alpha h^{\prime}\nn\\
&=\frac{1}{2}(\frac{p_\mu
}{c})_h p_\mu +(-\frac{3}{2}c(\frac{p_\mu }{c})_h-
\frac{c_h}{2c}p_\mu )\alpha-\frac{c_h}{2}\alpha^2.
\label{GE decouple 1}
\end{align}
Now, by (\ref{GE p h}),
\[
   -\frac{3}{2}c(\frac{p_\mu}{c})_h- \frac{c_h}{2c}p_\mu
  =-c(\frac{p_\mu }{c})_h-\frac{p_{\mu h}}{2}
  =-c(\frac{p_\mu }{c})_h-\frac{c_\mu }{2},
\]
and, since $\mu'=c$, we have
\[
  \sqrt{c}\alpha^\prime+\frac{c_h}{2\sqrt{c}}\alpha h^\prime+
  \frac{\sqrt{c}}{2}c_\mu \alpha=(\sqrt{c}\alpha)^\prime.
\]
Thus, multiplying \eq{GE decouple 1} by $\sqrt{c}$ and simplifying, we
get
\begin{align}
  (\sqrt{c}\alpha)^\prime-\frac{1}{2}\sqrt{c}&
   (\frac{p_\mu }{c})_h h^{\prime}\nn\\
  &=\frac{1}{2}\sqrt{c}(\frac{p_\mu }{c})_h p_\mu
  -c\sqrt{c}(\frac{p_\mu}{c})_h\alpha-\frac{c_h\sqrt{c}}{2}\alpha^2.
\label{GE decouple2}
\end{align}

By (\ref{GE y def}),
\[
  \alpha=\frac{y+I}{\sqrt{c}},
\]
where $I$ is defined in (\ref{GE Idef}) and satisfies
\[
   I^\prime=I_h h^\prime+c I_\mu.
\]
Using these in (\ref{GE decouple2}), we get
\beq
   y^\prime=a_0+a_1 y-a_2 y^2,
\label{GE decouple3}
\eeq
where $a_0$, $a_1$ and $a_2$ are defined in \eq{GE a0}--\eq{GE a2}. 

The derivation of the differential equation along backward
characteristics,
\beq 
   q^\backprime=a_0 - a_1 q - a_2 q^2,
\label{GE decouple4}
\eeq
where $q$ is defined in (\ref{GE q def}), is similar, and the proof is
complete.
$\square$\bigskip

\begin{cor}
\label{ge cor 2}
A singularity (gradient blowup) forms if and only if
\[
   |y|\ \text{or}\ |q|\to \infty \com{iff}
   |u_x|\ \text{or}\ |v _x| \to \infty,
\]
provided the solution takes values in the compact set $\mathcal K$.
\end{cor}

\begin{proof}
By (\ref{GE y def}) and (\ref{GE q def}),
\begin{align*}
   y+q&=\sqrt{c}(\alpha+\beta),\\
   y-q&=\sqrt{c}(\alpha-\beta)-2I.
\end{align*}
By (\ref{GE c pm c h}), (\ref{GE Idef}) and compactness, $I$ remains
finite, and the result follows.
\end{proof}

Because the coefficients $a_0$, $a_1$, and $a_2$ in Theorem \ref{GE
  them new ODEs} don't include derivative terms $v_x$, $u_x$, $v_t$ or
$u_t$, they are lower order when compared to $y$ and $q$.  Using
(\ref{GE xiMtaux}) and (\ref{GE xiZtaux}), these coefficients can be
expressed in terms of $v$ and $\barx$ rather than $h$ and $\mu$, as in
Corollary \ref{GE cor change coordinate}.

In the $p$-system, $p=p(v)$ and $p_\mu=0$, so that $a_0=a_1=0$.  In
this case, (\ref{GE new ode1}) and (\ref{GE new ode2}) become
\beq
   y^\prime= -a_2 y^2,\quad q^\backprime= -a_2q^2,
\label{ODEs psystem}
\eeq
which are exactly the ODEs derived in \cite{lax2,lax3} for the
isentropic homogeneous case.

\section{Formation of singularity}

We now consider the formation of singularities, which take the form of
the blowup of gradients $u_x$ and/or $v_x$, and correspond to shock
formation in conservative systems.

We will study the equations  (\ref{GE new ode1}), (\ref{GE new ode2})
as a dynamical system, even though this is not a pure system of ODEs.
First, consider the ODE 
\[
    \dot \xi = \psi_\pm(\xi),
\]
where $\psi_\pm$ is defined by
\beq
  \psi_\pm(\xi) \equiv a_0 \pm a_1 \,\xi - a_2 \,\xi^2,
\label{GE xi}
\eeq
with $a_2>0$, and the $a_i$ are treated as constants.  The equilibria,
if they exist, are the roots of the quadratic equation
\[
  a_0 \pm a_1 \,\xi - a_2 \,\xi^2 = 0,
\]
and we have $\dot\xi>0$ between the roots, and $\dot\xi<0$ otherwise.
A typical phase line is shown in Figure~\ref{GE figure}.

\begin{figure}[htp] \centering
\begin{tikzpicture}
  \draw[->] (-3,0) -- (3,0) node[below] {$\xi$};
  \draw (-1.5,-1.25) parabola bend (0,1) (1.5,-1.25) 
       node[right] {$\psi^\nu_\pm$};
  \fill (-1,0) circle (0.05) node[above left] {$\xi_1$};
  \fill (1,0) circle (0.05) node[above right] {$\xi_2$};
  \draw (-0.1,-0.1) -- (0.1,0) -- (-0.1,0.1);
  \draw (-1.9,-0.1) -- (-2.1,0) -- (-1.9,0.1);
  \draw (2.1,-0.1) -- (1.9,0) -- (2.1,0.1);
\end{tikzpicture}
\caption{Phase line for $\dot\xi = \psi^\nu_\pm(\xi)$.}
\label{GE figure}
\end{figure}
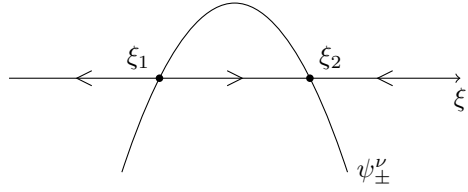

Whenever there are real roots, say $\xi_1\le\xi_2$, the region
$\{\xi>\xi_1\}$ is an invariant region for the ODE.  In particular, if
$a_0>0$, then the roots are of opposite signs and the region
$\{\xi>0\}$ is invariant.  Moreover, the region $\{\xi<\xi_1\}$ is
also invariant, and solutions that originate in this interval have
only a finite time of existence: it is this that drives the gowth and
blowup of gradients in the full system \eq{GE decouple3}, \eq{GE
  decouple4}.

For the $p$-system, we have $p_\mu =0$, so $a_0=a_1=0$ and $I=0$.  In
this case, $y=\sqrt{c}\,s_x$ and $q=\sqrt{c}\,r_x$ are (multiples of)
the gradients of the Riemann invariants.  By the above discussion, the
regions $\{ y>0\}$, $\{q>0\}$ are invariant domains for the system
\eq{ODEs psystem}, as are the regions $\{ y<0\}$, $\{q<0\}$.  Thus if
$y$ or $q$ is negative somewhere, the negative quadratic functions in
(\ref{ODEs psystem}) drive blowup of $y$ or $q$ in finite time, as
shown by Lax in \cite{lax2}.  From our point of view this is a trivial
case, for which we have uniform estimates.  In general, the forward
and backward waves interact nonlinearly with the varying stationary
background, and we do not expect uniformity.  In fact, interactions
can cause waves to change their \RC character, as demonstrated in
\cite{young blake 1} and the upcoming paper \cite{G4}.  We expect that
a complete analysis of the dynamics of \eq{GE new ode1}, \eq{GE new
  ode2} will yield a rich variety of new and unexpected phenomena.

We now prove the breakdown results of Theorem~\ref{GE Thm
  singularity2} by studying \eq{GE decouple3}, \eq{GE decouple4} as a
dynamical system.  Our aim is to describe sufficient conditions which
imply that the gradient blows up in finite time.

{\it Proof of Theorem \ref{GE Thm singularity2}}.\quad
Fix a constant $0<\nu\ll 1$, and define
\beq
  \psi^\nu_\pm(\xi) \equiv a_0 \pm a_1 \,\xi - (1-\nu)\,a_2 \,\xi^2,
\label{GE psi}
\eeq
so that our ODEs \eq{GE decouple3}, \eq{GE decouple4} can be written
\beq
   y^\prime = \psi^\nu_+(y) - \nu\,a_2\,y^2 \com{and}
   q^\bp =\psi^\nu_-(q) - \nu\,a_2\,q^2.
\label{GE psi eqs}
\eeq

Now let $N=N(\nu)<0$ be a uniform lower bound for the (real) roots of
$\psi^\nu_\pm$, or $N=0$ if there are no real roots.  Then since
$a_2>0$, we have 
\[
  \psi^\nu_\pm(\xi) \le 0 \com{for every}
  \xi \le N.
\]

Now suppose there is some $x_0$ such that the data satisfies
\[
  y_0 = y(x_0,0) < N.
\]
Then for the forward characteristic emanating from $(x_0,0)$, we have
\[
  y' \le - \nu\, a_2\,y^2,
\]
so that the solution satisfies
\beq 
   \frac{1}{y(t)}\ge{\frac{1}{y(0)}+\nu\int_{0}^t{a_2}\,dt},
\label{GE ysol}
\eeq
where the integral is taken along the forward characteristic.
Since $y(0)<0$ and $a_2$ is uniformly positive, there is some finite
$T_*$ such that the right hand side of \eq{GE ysol} vanishes, and so
we obtain
\[
  y(t) \to -\infty \com{as} t \to T_*.
\]
A similar calculation holds for $q$.

Finally, we calculate the lower bound $N$.  The roots of
$\psi^\nu_\pm$ solve the quadratic equations
\[
  \psi^\nu_\pm(\xi) = a_0 \pm a_1 \,\xi - (1-\nu)\,a_2 \,\xi^2 = 0
\]
and so are
\[
   \xi = \frac{\pm a_1 \pm \sqrt\Delta}{2\,(1-\nu)\,a_2},
\com{where}
   \Delta = a_1^2 + 4\,(1-\nu)\,a_0\,a_2.
\]
The minimum of these is clearly
\[
   -\frac{ |a_1| + \sqrt\Delta}{2\,(1-\nu)\,a_2},
\]
and the lower bound $N$ is obtained by maximizing the ratios
$|a_1|/a_2$ and $a_0/a_2$, so we look for upper bounds for $a_0$ and
$|a_1|$, and a lower bound for $a_2$.

From the expressions \eq{GE a0}--\eq{GE a2} and \eq{GE Idef}, we
obtain the bounds provided $c$ and $c_h$ are bounded away from zero,
while the quantities
\[
  h,\ c,\ c_h,\ p_\mu,\ p_{\mu h},\ p_{\mu hh},\com{and} p_{\mu\mu h},
\]
remain finite, and recalling \eq{GE p h}, the bound follows.

Using (\ref{GE xiMtaux}) and (\ref{GE xiZtaux}), the bounds can also
be expressed by the bounds on the derivatives of $p$ (and $c$) with
respect to $v$ and $\barx$.
$\square$\bigskip

When $a_0>0$, the equilibria are on opposite sides of the origin, so
the stable invariant region includes the values $\xi_1<\xi\le 0$.  If
the data can be set up in such a way that $y_0$ and $q_0$ lie in the
interval $(\xi_1,\infty)$ for all $x_0$, then the corresponding
solution would have finite gradients for all time, and these would be
nontrivial shock-free solutions.  This topic is part of the authors'
ongoing research.


\section{Generalized \RC structure}

The \RC structure at a single contact discontinuity or entropy jump
for the compressible Euler equation is fully analyzed in \cite{young
  blake 1}.  This is an analysis of how the \RC character changes when
a wave crosses a jump discontinuity, where the system is a $p$-system
on either side of the jump.  Using this \RC structure on an entropy
jump, a class of time-periodic solutions in compressible Euler
equations with polytropic ideal gas has been studied in a series of
papers \cite{young blake 1,young blake 2,young blake 3}. In this
section, we show that the \RC characters we have defined for smooth
solutions are consistent with those results.  Furthermore, by the
study of \RC structures in the generalized Euler equations (\ref{GE
  lagrangian1}), (\ref{GE lagrangian2}) with smooth and piecewise
smooth entropy fields, we expect eventually to see a large class of
time-periodic or quasi-periodic solutions as in \cite{young blake
  1,young blake 2} with both piecewise smooth and smooth entropy
profiles.

We first consider how the \RC character can change at a fixed point in
a smooth solution.

\begin{lem}
If $(\frac{p_{\barx}}{p_v})_v<0$ (or equivalently
$(\frac{p_\mu}{c})_h<0$), the backward \RC character can only change
from $C$ to $R$ (resp.~$R$ to $C$), if the solution is forward $C$
(resp.~$R$); the forward \RC character can only change from $C$ to $R$
(resp.~$C$ to $R$), if the solution is backward $R$ (resp.~$C$).  If
$(\frac{p_{\barx}}{p_v })_v>0$ (or $(\frac{p_\mu}{c})_h>0$), all the
above \RC character changes will only happen in the opposite direction.
\end{lem}

\begin{proof}
Suppose the backward \RC character changes at $(x_0,t_0)$, then
$\beta=0$ at $(x_0,t_0)$.  Then by (\ref{GE rem2}), we have
\[
  \beta^\bp = \frac{c}{2}\,(\frac{p_\mu}{c})_h\;\alpha,
\]
which has the sign of $\alpha\,(\frac{p_\mu}{c})_h$.  Thus, the
backward wave changes from $C$ ($\beta<0$) to $R$ ($\beta>0$) at
$(x_0,t_0)$, if and only if $\alpha$ and $(\frac{p_\mu}{c})_h$ have
the same sign.  Thus if $(\frac{p_\mu}{c})_h<0$, then $\alpha<0$, and
the crossing forward wave is necessarily compressive by
Corollary~\ref{GE def2}.  All other cases follow similarly, and using
Corollary~\ref{GE cor change coordinate} for the equivalence of
derivative conditions completes the proof.
\end{proof}

In order to compare the \RC structure for smooth solutions to
that of a single entropy jump, we recall the relevant argument from 
\cite{young blake 1}.  Discontinuities in weak solutions are governed
by the Rankine-Hugoniot jump conditions, which are
\beq
  \sigma[v] = -[u],\com{and} 
  \sigma[u] = [p],
\label{GE RH}
\eeq
plus a third equation for energy conservation.  Here $\sigma$ is the
speed of the discontinuity, and brackets denote the jump $[f]=f_R-f_L$
in $f$ across the discontinuity.  An entropy jump has zero speed,
$\sigma=0$, so that \eq{GE RH} reduce to $[u]=0=[p]$, that is
\[
  u_R = u_L, \com{and}
  p_R = p_L.
\]
Across the entropy jump, the \RC structure changes as
follows, see \cite{young blake 1}: 

\begin{lem} 
For $\frac{c_R}{c_L}<1$, the backward \RC character can only change
from $R$ to $C$ (resp.~$C$ to $R$), if the solution is forward
$R$(resp.~$C$); the forward \RC character can only change from $C$ to
$R$ (resp.~$R$ to $C$), if the solution is backward
$R$(resp.~$C$). For $\frac{c_R}{c_L}>1$, all the above \RC character
changes will only happen in the opposite direction.
\end{lem}

Comparing these Lemmas, we see that the condition $\frac{c_R}{c_L}<1$
for an entropy jump should be consistent with the condition
$(\frac{p_{\barx}}{p_v})_v<0$ for smooth solutions.  Recalling that
the change in pressure at an entropy jump is zero, we can assume
$p_x=0$, and treat the condition $\frac{c_R}{c_L}<1$ as $c_x<0$.
Thus it suffices to show that, if $p_x=0$, then
\beq
  (\frac{p_{\barx}}{p_v })_v <0\Leftrightarrow c_x<0.
\label{GE RCequi}
\eeq

Since $f_x = f_v\,v_x+f_{\barx}$, if $p_x=0$, we have
\[
   v_x = -\frac{p_{\barx}}{p_v},
\]
so that also
\[
  c_x = c_{v }v _x+c_{\barx} 
   = -c_{v}\frac{p_{\barx}}{p_v}+c_{\barx}
   = \frac{-c_{v }p_{\barx}+c_{\barx} p_v}{p_v}.
\]
On the other hand,
\[
  (\frac{p_{\barx}}{p_v})_v 
    = \frac{p_{\barx v }p_v - p_{\barx}p_{vv}}{(p_v)^2}
    = - \frac{2c(c_{\barx}p_v -p_{\barx}c_{v})}{(p_v )^2},
\]
where we recall $p_v = -c^2$.  Comparing these, it follows that
\[
  (\frac{p_{\barx}}{p_v})_v = - \frac{2c}{p_v}\;c_x = \frac2c\;c_x,
\]
which proves (\ref{GE RCequi}).  Thus our continuous \RC character is
consistent with that of a single entropy jump.

\section{Transverse flow of 1D MHD}

In this section, we consider the motion of a compressible fluid
coupled a magnetic field $H=(H_1, H_2, H_3)$, which satisfy the
Magnetohydrodynamic (MHD) equations in Lagrangian coordinates. In
one space dimension, the MHD equations can be written in the
non-conservative form
\begin{align} 
\frac{\partial v }{\partial t}-\frac{\partial u_1}
{\partial x} &= 0\label{GE mhd1},
\\\frac{\partial H_2}{\partial t}+\rho H_2 \frac{\partial u_1}
{\partial x}-
\rho H_1 \frac{\partial u_2}{\partial x} &= 0,\label{GE mhd2}\\
\frac{\partial H_3}{\partial t}+\rho H_3 \frac{\partial
u_1}{\partial x}- \rho H_1 \frac{\partial u_3}{\partial x} &= 0,
\label{GE mhd3}\\
\frac{\partial u_1}{\partial t}+\frac{\partial }{\partial
x}(p+\frac{1}{2}\mu_0(H^2_2+H^2_3)) &= 0,\label{GE
mhd4}\\\frac{\partial
u_2}{\partial t}-\mu_0 H_1\frac{\partial H_2}{\partial x} &= 0,
\label{GE mhd5}\\
\frac{\partial u_3}{\partial t}-\mu_0 H_1\frac{\partial
H_3}{\partial x} &= 0,\label{GE mhd6}\\ \frac{\partial S}{\partial
t} &= 0,\label{GE mhd7}
\end{align}
where  $\rho$ is the density, $v =\rho^{-1}$ is the specific volume,
$(u_1,u_2,u_3)$ is the velocity field, $S$ is the entropy, and
$p=p(v ,S)$ is the pressure.  The permeability $\mu_0$ is a positive
constant, and $H_1$ is constant in the one-dimesional model,
c.f.~\cite{Li Qin,Rammaha}.

The mathematical structures of (\ref{GE mhd1})--(\ref{GE mhd7})
are totally different when $H_1$ is zero or a nonzero constant.  Here
we briefly consider the easier case $H_1=0$, which physically means
that the magnetic field is transverse to the direction of motion.  In this
case, (\ref{GE mhd1})--(\ref{GE mhd4}) and (\ref{GE mhd7}) form
a closed system~\cite{Li Qin}:
\begin{align}
\frac{\partial v }{\partial t}-\frac{\partial u_1}
{\partial x} &= 0\nn,\\\frac{\partial u_1}{\partial
t}+\frac{\partial \tilde{p}}{\partial x} &= 0,\nn\\
\frac{\partial \tilde{H_2}}{\partial t} &= 0,\label{MHDsmall}\\
\frac{\partial \tilde{H_3}}{\partial t} &= 0,\nn\\
\frac{\partial S}{\partial t} &= 0,\nn
\end{align}
where
\[
  \tilde{H_2}=v  H_2,\quad
  \tilde{H_3}=v  H_3,\quad
  \tilde{p}=p+\frac{1}{2}\mu_0
     \frac{{\tilde{H}}^2_2+{\tilde{H}}^2_3}{v^2}.
\]
Thus in the special case of transverse flow, the one-dimensional MHD
equations fit into the framework of (\ref{GE lagrangian1}), (\ref{GE
lagrangian2}), and our previous results apply directly.

We analyze this further for a polytropic ideal ($\gamma$-law) gas, so
that the pressure is given by
\[
   p=K\,e^{{S}/{c_v}}\,v ^{-\gamma},
\]
with adiabatic gas constant $\gamma>1$, and $K$, $c_v$ are positive
constants.  Then we get
\begin{align*}
  \tilde{p} &= K\,e^{{S}/{c_v}}\,v^{-\gamma}+\frac{1}{2}\mu_0
     \frac{{\tilde{H}}^2_2+{\tilde{H}}^2_3}{v^2} \\
  &= A_1\,v^{-\gamma} + A_2\, v^{-2},
\end{align*}
where we have set
\[
  A_1(x) \equiv Ke^{{S(x)}/{c_v }} \com{and}
  A_2(x) \equiv \frac{1}{2}\mu_0 (\tilde{H}^2_2(x)+{\tilde{H}}^2_3(x)).
\]

In order to further simplify our calculation, we now assume
$\gamma=2$, and set
\[
   B(x)\equiv A_1(x)+A_2(x), \com{so that}
   \tilde p(v,x) = B(x)\,v^{-2}.
\]
We then calculate
\[
  c(v,x) = \sqrt{-p_v} = \sqrt{2}\,\sqrt{B(x)}\,v^{-3/2},
\]
and, from \eq{hmudef},
\[
  h(v,x) = \int_v^\infty c\;dv 
         = 2\sqrt2\,\sqrt{B(x)}\,v^{-1/2}.
\]
Now solving for $v$ and recalling $\mu = x$, we obtain
\[
  v(h,\mu) = 8\,B(\mu)\,h^{-2},
\]
and so also
\[
  p(h,\mu) = \frac{h^4}{64\,B(\mu)}  \com{and}
  c(h,\mu) = \frac{h^3}{16\,B(\mu)}.
\]
Note that
\[
  \frac{\partial p}{\partial h}(h,\mu) = c(h,\mu) \com{and}
  \frac{\partial v}{\partial h}(h,\mu) = \frac{-1}{c(h,\mu)},
\]
as intended by the choice of $h$, see~\cite{young Global
  interaction}.  Next, we calculate
\[
  \frac{p_\mu}{c} = -\frac{h}{4}\,\frac{\dot B(\mu)}{B(\mu)},
\]
where $\dot B \equiv \frac{dB}{d\mu}$, so that
\[
  \alpha = u_x + h_x - \frac{h}{4}\,\frac{\dot B(x)}{B(x)} \com{and}
  \beta = u_x - h_x + \frac{h}{4}\,\frac{\dot B(x)}{B(x)};
\]
here $\alpha$ and $\beta$ are evaluated at the point $(x,t)$, and $\mu=x$.
From \eq{GE Idef}, we calculate
\beq
  I(h,\mu) = \frac{-1}{80}\,\frac{\dot B(\mu)}{B(\mu)^{3/2}}\,h^{5/2},
\label{I}
\eeq
which also leads to
\[
  \frac{p_\mu}{\sqrt c} - I 
   = \frac{-1}{20}\,\frac{\dot B(\mu)}{B(\mu)^{3/2}}\,h^{5/2},
\]
and thus, by \eq{GE y def}, \eq{GE q def},
\[
  y = \frac{h^{3/2}}{4\sqrt B}\;
     \left(u_x + h_x - \frac{h}{5}\,\frac{\dot B}{B}\right),
\com{and}
  q = \frac{h^{3/2}}{4\sqrt B}\;
     \left(u_x - h_x + \frac{h}{5}\,\frac{\dot B}{B}\right).
\]
Finally, using \eq{GE a0}--\eq{GE a2}, we calculate the coefficients
\[
  a_2 = \frac38\,\frac{h^{1/2}}{\sqrt B}, \com{and}
  a_1 = \frac{1}{40}\,\frac{\dot B}{B^2}\,h^3 
      = -2\,\frac{h^{1/2}}{\sqrt B}\,I,
\]
where we have used \eq{I}, and, after simplification,
\begin{align*}
  a_0 &= h^{11/2}\;\left[ \frac1{2^8\,5}\,\frac1B\,
      \Big(\frac{\dot B}{B^{3/2}}\Big)\dot{\vphantom{\Big|}}
      +\frac{\dot B^2}{B^{7/2}}\,\Big(
        \frac1{2^{11}}-\frac1{2^{10}\,5}-\frac3{2^{11}\,5^2}
      \Big)\right]\\
     &= \frac{h^{11/2}}{2^8\,5}\;\left[
         \frac{\ddot B}{B^{5/2}} - \frac65\frac{\dot
           B^2}{B^{7/2}}\right]
       = - 6\,\frac{h^{1/2}}{\sqrt B}\,G(\mu)\,I^2,
\end{align*}
where we have set
\[
   G(\mu) \equiv 1 - \frac56\,\frac{B(\mu)\;\ddot B(\mu)}{\dot
     B(\mu)^2},
\]
wherever $\dot B(\mu) \ne 0$.

We use these in \eq{GE xi} to calculate the quadratic
\[
  \psi_\pm(\xi) = - \frac{h^{1/2}}{\sqrt B}\;\left(
       6\,G\,I^2 \pm 2\,I\,\xi + \frac38\,\xi^2\right),
\]
and the corresponding dynamical system can be analyzed as above.  It
is clear that the growth of the quantities $y$ and/or $q$ depends
critically on their size relative to $I$.

\section{Euler flow in a variable area duct}

Finally, we consider the compressible Euler flow through a duct of
varying cross section $a(\tilx)$. In (spatial) Eulerian
coordinates $(\tilx,\tilt)$, this flow satisfies:
\begin{align}
  a_{\tilt}&=0,\label{duct1}\\
  (a\rho)_{\tilt}+(a\rho u)_{\tilx}&=0,\label{duct2}\\
  (a\rho u)_{\tilt}+(a\rho u^2)_{\tilx}+a p_{\tilx}
       &=0,\label{duct3}\\
  (a\rho E)_{\tilt}+(a\rho E u+apu)_{\tilx}=0,\label{duct4}
\end{align}
where $E$ is the energy, $\rho$ is the density, $u$ is the velocity,
$p$ is the pressure and $a=a(\tilx)$ is the diameter of the duct
at position $\tilx$, \cite{courant, Dafermos, Hong Temple}. For
smooth solutions, the energy equation (\ref{duct4}) is equivalent to
\beq
   S_{\tilt}+u S_{\tilx}=0,
\label{ductS}
\eeq
where $S$ is the entropy \cite{courant}. By introducing Lagrangian
coordinates $(x,t)$, which satisfy
\begin{align}
  dx &= a\rho\, d\tilx- a\rho u\, d \tilt,
\label{duct transform1}\\
  dt &= d\tilt,
\label{duct transform2}
\end{align}
it is easy to check the smooth solutions of
(\ref{duct1})--(\ref{duct3}) and (\ref{ductS}) satisfy
\begin{align}
  \hat{{v}}_t-u_x&=0,\label{duct l2}\\
  u_t+a p_x&=0,\label{duct l3}\\
  S_t&=0, \label{duct l4}
\end{align}
where $\hat v$ is the specific volume per unit cross-sectional area,
\[
  \hat{{v}}=\frac{1}{a\rho}.
\]
By  by (\ref{duct2}), the right hand side of (\ref{duct transform1})
is an exact differential, so there is no difficulty defining $x$.  By
(\ref{duct transform1}), (\ref{duct transform2}), $a(x,t)$ satisfies
the identities
\beq
   a_t=u\,\dot{a},\quad
   a_x=\hat{{v}}\,\dot{a},\quad
   (\dot{a})_t=u\,\ddot{a},\com{and}
   (\dot{a})_x=\hat{{v}}\,\ddot{a},
\label{duct l1}
\eeq
where we denote
\[
   \dot{a}=\frac{d a(\tilx)}{d\tilx},\com{and}
   \ddot{a}=\frac{d^2 a(\tilx)}{d{\tilx}^2},
\]
these describing the changing shape of the duct in spatial coordinates.

We again assume we have an ideal polytropic gas, so that
\[
   p = K\,e^{S/c_v}\,v^{-\gamma}
     = K\,e^{S/c_v}\,(a\,\hat{{v}})^{-\gamma}.
\]
with adiabatic gas constant $\gamma>1$, c.f.~\cite{courant}.  The
(Lagrangian) sound speed is
\[
  c = \sqrt{-a\,p_{\hat{{v}}}}
    = \sqrt{K\gamma}\,a^{-\frac{\gamma-1}{2}}\;
      {\hat{{v}}}^{-\frac{\gamma+1}{2}}\,e^{\frac{S}{2c_v}}.
\]

Following \cite{young blake 1}, we make the change of variables
\beq
   m = e^{S/2c_v}   \com{and}
   z = \int^\infty_{\hat v}
       \frac{c}{a^{-\frac{\gamma-1}2}\,m}\;d\hat v
     = \frac{2\sqrt{K\gamma}}{\gamma-1}\,
       \hat v^{-\frac{\gamma-1}{2}}.
\label{z def}
\eeq
It follows that
\begin{align*}
   \hat{{v}} &= K_{{v}}z^{-\frac{2}{\gamma-1}}, \\
   p &= K_p a^{-\gamma}m^2 z^{\frac{2\gamma}{\gamma-1}}, 
\com{and}\\
   c &= c(z,m)=K_c a^{-\frac{\gamma-1}{2}}m
   z^{\frac{\gamma+1}{\gamma-1}},
\end{align*}
where $K_{v}$, $K_p$ and $K_c$ are appropriate positive constants,
see~\cite{young blake 1}.

For $C^1$ solutions, the {Lagrangian} equations (\ref{duct
  l2})--(\ref{duct l4}) are equivalent to
\begin{align}
   z_t+\frac{c}{a^{-\frac{\gamma-1}{2}}m}u_x&=0,
\label{duct lagrangian1 zm}\\
   u_t+mca^{-\frac{\gamma-1}{2}}z_x+2\frac{ap}{m}m_x-\gamma p a_x&=0,
\label{duct lagrangian2 zm}\\\label{duct lagrangian3 zm}
   m_t&=0.
\end{align}

We define
\begin{align*}
   \alpha &= u_x+a^{-\frac{\gamma-1}{2}}mz_x+\frac{\gamma-1}{\gamma}
       a^{-\frac{\gamma-1}{2}}m_x z \com{and}\\
   \beta &= u_x-a^{-\frac{\gamma-1}{2}}mz_x-
       \frac{\gamma-1}{\gamma} a^{-\frac{\gamma-1}{2}}m_x z,
\end{align*}
which are the gradient variables used for the polytropic ideal Euler
flow in \cite{G3}, adjusted to account for $a(\tilx)$.  By
similar calculations as in the proof of Theorem \ref{GE remark} or in
\cite{G3}, we obtain
\begin{align}
  \alpha^\prime&=k_1(k_2(3\alpha+\beta)+(\alpha\beta-\alpha^2))+k_3
   (\alpha-\beta)+A(x,t),
\label{duct alpha}\\
  \beta^\bp &=-k_1(k_2(3\beta+\alpha)+(\alpha\beta-\beta^2))+k_3
   (\beta-\alpha)-A(x,t),
\label{duct beta}
\end{align}
where
\[
  A(x,t) = \frac{(\gamma-1)^3}{8K_c}m^2z^{\frac{2\gamma-4}{\gamma-1}}
a^{-\gamma-1}[a \ddot{a}-\gamma
\dot{a}^2]+\frac{(\gamma-1)^2}{2\gamma}m
 m_x z^2 a^{-\gamma}\dot{a},
\]
and the coefficients are
\begin{align*}
  k_1 &= \frac{\gamma+1}{2(\gamma-1)}K_c z^{\frac{2}{\gamma-1}},\\
  k_2 &= \frac{\gamma-1}{\gamma{\gamma+1}}m_x z
     a^{-\frac{\gamma-1}{2}}, \com{and}\\
  k_3 &= \frac{3(\gamma-1)^3}{8}m z  a^{-\frac{\gamma+1}{2}}
     \dot{a}-\frac{\gamma-1}{4}u a^{-1}\dot{a}.
\end{align*}
By (\ref{duct lagrangian1 zm})--(\ref{duct lagrangian3 zm}),
\[
  z^{\prime} = z_t+c\,z_x = -K_c
   z^{\frac{\gamma+1}{\gamma-1}}(\beta+\frac{\gamma-1}{\gamma}m_x z
   a^{-\frac{\gamma-1}{2}}),
\]
so that
\beq
   \beta= -\frac{z^{\prime}}
   {K_c z^{\frac{\gamma+1}{\gamma-1}}}-\frac{\gamma-1}{\gamma}m_x z
   a^{-\frac{\gamma-1}{2}}.
\label{dbeta z prime}
\eeq

We plug (\ref{dbeta z prime}) into (\ref{duct alpha}), move
terms including $z^\prime$ to the left hand side, and multiply
by $z^{\frac{\gamma+1}{2(\gamma-1)}}$ on both sides, to obtain
\begin{align}
(z^{\frac{\gamma+1}{2(\gamma-1)}}&\alpha)^{\prime}+
\frac{m_x}{2\gamma}a^{-\frac{\gamma-1}{2}}z^{\frac{\gamma+1}
{2(\gamma-1)}}
z^{\prime}-\frac{1 }
{K_c }k_3 z^{-\frac{\gamma+1}{2(\gamma-1)}}z^{\prime}\nn\\
&=k_1(k_2(3\alpha-{\textstyle\frac{\gamma-1}{\gamma}}m_x z
a^{-\frac{\gamma-1}{2}})+(-{\textstyle\frac{\gamma-1}{\gamma}}m_x z
a^{-\frac{\gamma-1}{2}}\alpha-\alpha^2))z^{\frac{\gamma+1}
{2(\gamma-1)}}
\nn\\&\quad{}+k_3(\alpha+{\textstyle\frac{\gamma-1}{\gamma}}m_x
z a^{-\frac{\gamma-1}{2}})
z^{\frac{\gamma+1}{2(\gamma-1)}}+A(x,t)z^{\frac{\gamma+1}
{2(\gamma-1)}}.
\label{duct y eqn}
\end{align}
Furthermore, by (\ref{duct lagrangian1 zm})--(\ref{duct lagrangian3
  zm}), we have 
\begin{align}
   u^{\prime} &= u_t+c u_x = c\beta+\gamma p a_x\nn\\ 
    &= -m a^{-\frac{\gamma-1}{2}}z^\prime-\frac{\gamma-1}{\gamma}c
    z m_x a^{-\frac{\gamma-1}{2}}+\gamma p a_x.\label{duct u prime}
\end{align}
By (\ref{duct u prime}) and following the proof of Theorem \ref{GE
them new ODEs}, (\ref{duct y eqn}) can again be transformed into
\[
   y^\prime= \psi_+(y) = a_0+a_1 y-a_2 y^2,
\]
where $y$ is a gradient variable, and $a_0$, $a_1$ are functions of
$m$, $m_x$, $m_{xx}$, $z$, $a$, $\dot{a}$, $\ddot{a}$, and $u$, and 
\[
  a_2 = \frac{\gamma+1}{2(\gamma-1)}K_c
  z^{\frac{\gamma+1}{2(\gamma-1)}-1} > 0.
\]
Similar considerations for backward waves yield a similar equation for
the gradient variable $q$, namely
\[
   q^\bp = \psi_-(q) = a_0-a_1\, q-a_2\, q^2.
\]
A blowup result similar to Theorem~\ref{GE Thm singularity2} can again
be obtained for this system.  We omit the details since the
calculations are tedious but similar to our previous calculation, and
$a_0$ and $a_1$ are more complicated expressions.  Note in particular
that they depend on $u$ as well as $z$, which is a new feature of this
case.

\end{document}